\documentclass[12pt]{article}

\usepackage{amsmath, amssymb, amsthm} 
\usepackage{geometry} 
\usepackage{graphicx} 
\usepackage{hyperref} 
\usepackage{cite} 
\usepackage{titlesec} 
\usepackage{multirow}%
\usepackage{amsmath,amssymb,amsfonts}%
\usepackage{amsthm}%
\usepackage{mathrsfs}%
\usepackage[title]{appendix}%
\usepackage{xcolor}%
\usepackage{textcomp}%
\usepackage{manyfoot}%
\usepackage{booktabs}%
\usepackage{algorithm}%
\usepackage{algorithmicx}%
\usepackage{algpseudocode}%
\usepackage{listings}%
\usepackage{enumitem}

\titleformat{\section}{\normalfont\Large\bfseries}{\thesection}{1em}{}

\newtheorem{theorem}{\textbf{Theorem}}[section]
\newtheorem{proposition}[theorem]{\textbf{Proposition}}
\newtheorem{lemma}[theorem]{\textbf{Lemma}}

\theoremstyle{definition}
\newtheorem{definition}[theorem]{Definition}

\newtheorem{example}[theorem]{Example}

\theoremstyle{remark}
\newtheorem{remark}[theorem]{Remark}

\begin{document}

\title{Statistical Biharmonicity of Identity Maps}
\author{Ryu Ueno}

\date{}

\maketitle

\begin{abstract}
    A statistical structure is a pair of a Riemannian metric and a torsion-free affine connection that satisfies the Codazzi equation, and a manifold with a statistical structure is called a statistical manifold. 
    In a certain setting, the tension field of the identity map between two statistical manifolds of different statistical structures is a scalar multiple of the Tchebychev vector field, which is an important object in affine differential geometry.
    We observe the statistical biharmonicity of this identity map.
    In such a setting, we determine statistical manifolds that has this identity map as a biharmonic map.
\end{abstract}

\makeatletter
\renewcommand{\@makefntext}[1]{\noindent#1}
\makeatother
\renewcommand{\thefootnote}{}
\footnotetext{The author is supported by JST SPRING, Grant Number JPMJSP2119.}

\section*{Introduction} 
Statistical manifolds have been widely studied since their introduction in information geometry in 1982 \cite{Am}. However, they have also appeared in affine differential geometry.
Given a locally strongly convex hypersurface immersion in real affine space, a statistical structure and its Tchebychev vector field are induced on the manifold from a chosen equiaffine transversal vector field.
In centroaffine geometry, the shape operator of the hypersurface is given by the Tchebychev vector field\cite{Liu1995TheCT}, and in equiaffine geometry, the transversal vector field is chosen such that the Tchebychev vector field vanishes \cite{NomizuSasaki}. 
Generally, statistical manifolds with vanishing Tchebychev vector fields are said to satisfy the \textit{equiaffine condition}.
The Tchebychev vector field has also been studied in the context of applications in information geometry \cite{MR2280051}.

Riemannian manifolds are the simplest statistical manifolds. 
In Riemannian geometry, the tension field is induced from the first variational formula of the energy functional for mappings between Riemannian manifolds, and mappings with vanishing tension fields are called harmonic maps. 
For example, the identity map on a Riemannian manifold is a trivial harmonic map.
Furthermore, as a generalization of harmonic maps, the notion of \textit{biharmonic map} was introduced by J. Eells and L. Lemaire in 1983 as a critical point of the bi-energy functional for mappings between Riemannian manifolds\cite{EL}. 
    
The notion of the tension field can be extended to mappings from a manifold equipped with a Riemannian metric and an affine connection to a manifold with an affine connection.
On the other hand, a \textit{statistical biharmonic map} must be defined between
two statistical manifolds, which is a class of mapping induced by a variational problem\cite{FU}. 
In particular, a mapping between statistical manifolds of vanishing tension fields is a trivial statistical biharmonic map. 

Let $(M,g,\nabla)$ be a statistical manifold, $T$ the Tchebychev vector field, $\overline{\nabla}$ the conjugate connection of $(M,g,\nabla)$, and $\nabla^g$ the Levi-Civita connection of $g$.
If we denote the tension fields of $\mathrm{id}:(M,g,\nabla)\to(M,g,\nabla^g)$, $\mathrm{id}:(M,g,\overline{\nabla})\to(M,g,\nabla^g)$ by $\tau(\mathrm{id})$, $\overline{\tau}(\mathrm{id})$, respectively, then $\tau(\mathrm{id})=-T$ and $\overline{\tau}(\mathrm{id})=T$ hold. 
Thus, the statistical manifold satisfies the equiaffine condition if and only if equation $\tau(\mathrm{id})=\overline{\tau}(\mathrm{id})=0$ holds. 
Let the statistical bi-tension field of the identity maps $\mathrm{id}:(M,g,\nabla)\to(M,g,\nabla^g)$, $\mathrm{id}:(M,g,\overline{\nabla})\to(M,g,\nabla^g)$ be $\tau_2(\mathrm{id})$, $\overline{\tau}_2(\mathrm{id})$, respectively. 
Our main objective is to prove in Theorem \ref{main1}, that the statistical manifold $(M,g,\nabla)$ satisfying the equation $\tau_2(\mathrm{id})=\overline{\tau}_2(\mathrm{id})=0$ is equivalent to the Tchebychev vector field $T$ satisfying the following equations (\ref{eq1}) and (\ref{eq2}).
\begin{equation}
    \label{eq1}
        \Delta_g T + \sum_{i=1}^m\mathrm{Ric}^g(e_i,T)e_i=0, \tag{T1}
\end{equation}
\begin{equation}
    \label{eq2}
        \mathrm{div}^g(T)T+\nabla^g_TT=0. \tag{T2}
\end{equation}
Therefore, statistical manifolds satisfying the equations (\ref{eq1}) and (\ref{eq2}) will be said to satisfy the \textit{semi-equiaffine condition}.

On a statistical manifold $(M,g,\nabla)$, the equations (\ref{eq1}) and (\ref{eq2}) imply that the Tchebychev vector field $T$ is a vector field that characterizes the space $(M,g)$ if it is non-zero. 
Thus in Theorem \ref{main2}, we determine statistical manifolds $(M,g,\nabla)$ of constant sectional curvature satisfying the semi-equiaffine condition, in the case where $(M,g)$ is a complete Riemannian manifold of constant curvature.


\section{Preliminaries}
\subsection{Statistical manifolds}

Throughout this paper, all the objects are assumed to be smooth.
Let $M$ be an orientable manifold of dimension $m \geq 2$, $C^{\infty}(M)$ the set of smooth functions on $M$, and $\Gamma(E)$ the set of sections of a vector bundle $E$ over $M$.

\vspace{0.5\baselineskip}

\begin{definition}
    Let $g$ be a Riemannian metric, $\nabla$ a torsion-free affine connection on $M$. 
    The pair $(g,\nabla)$ is called a statistical structure on $M$ if it satisfies the Codazzi equation:
    \begin{equation*}
        (\nabla g)(Y,Z;X)=(\nabla_X g)(Y,Z)=(\nabla_Yg)(X,Z),\quad X, Y, Z \in \Gamma(TM).
    \end{equation*}
    The triplet $(M,g,\nabla)$ is called a \textit{statistical manifold}.
\end{definition}
On a Riemannian manifold $(M,g)$, we denote by $\nabla^g$ the Levi-Civita connection. Here, since $\nabla^gg=0$, thus the triplet $(M,g,\nabla^g)$ is the simplest statistical manifold. 
In this case, the statistical manifold is called \textit{Riemannian}.

\begin{definition}
Let $(M,g,\nabla)$ be a statistical manifold.

$(1)$ \ 
Set $K=K^{(g,\nabla)} \in \Gamma(TM^{(1,2)})$ by
$$
K(X,Y)= K_XY= \nabla_X Y-\nabla^g_X Y,
\quad X, Y \in \Gamma(TM).   
$$
We also denote it by $K^M$ if necessary. 

$(2)$ \ 
Define an affine connection $\overline{\nabla}$ by
$$
Xg(Y,Z)=g(\nabla_XY,Z)+g(Y, \overline{\nabla}_XZ),
\quad X, Y, Z \in \Gamma(TM),    
$$
and call it the {\sl conjugate connection}\/ of $\nabla$ with respect to $g$. Conjugate connections are also called {\sl dual connections}\/. It is easy to check that $(M,g,\overline{\nabla})$ is also a statistical manifold.

$(3)$ \ 
Define the \textit{curvature tensor field} $R^\nabla \in \Gamma(TM^{(1,3)})$ 
of $\nabla$ by
$$
R^\nabla (X,Y)Z=\nabla_X \nabla_Y Z-\nabla_Y \nabla_X Z-\nabla_{[X,Y]}Z,
\quad X, Y, Z \in \Gamma(TM). 
$$
On a statistical manifold $(M,g,\nabla)$, we often denote $R^{\nabla}$ by $R$, 
$R^{\nabla^g}$ by $R^g$, and $R^{\overline{\nabla}}$ by $\overline{R}$, for short. 

$(4)$ \
Define the \textit{Ricci curvature tensor field} $\operatorname{Ric}^{\nabla}\in\Gamma(TM^{(0,2)})$ of $\nabla$ by
    $$
    \operatorname{Ric}^{\nabla}(X,Y)=\sum_{i=1}^mg(R^{\nabla}(e_i,X)Y,e_i),\quad X,Y\in\Gamma(TM),
    $$
where $\{e_1,\dots,e_m\}$ is an orthonormal frame on $M$ with respect to $g$. On a statistical manifold $(M,g,\nabla)$, we often denote $\operatorname{Ric}^{\nabla}$ by $\operatorname{Ric}$ and 
$\operatorname{Ric}^{\nabla^g}$ by $\operatorname{Ric}^g$, for short. 

$(5)$ \
Define a tensor field $L=L^{(g,\nabla)} \in\Gamma(TM^{(1,3)})$ by
    $$
    g(L(Z,W)X,Y) = g(R^\nabla (X,Y)Z,W),\quad X,Y,Z,W\in\Gamma(TM),
    $$
and call it the \textit{curvature interchange}  tensor field  
 of $(g,\nabla)$. We denote it $L^M$ if necessary. We also denote $\bar{L}=L^{(g, \overline{\nabla})}$.
\end{definition}

\begin{remark}
The following formulae hold for $X, Y, Z, W \in \Gamma(TM)$: 
\begin{eqnarray}
&&
\overline{\nabla}_X Y=\nabla^g_X Y-K(X,Y), \label{Kconjg}\\
&&
\nabla^g_X Y=\frac{\nabla_X Y+\overline{\nabla}_X Y}{2},\label{Leviconj}\\
&&
g(\overline{R}(X,Y)Z,W)=-g(Z, R(X,Y)W),\label{intcurvature}\\
&&
L(X,Y)Z+\overline{L}(X,Y)Z=R(X,Y)Z+\overline{R}(X,Y)Z.\label{sumcurv}
\end{eqnarray}
\end{remark}

\vspace{0.5\baselineskip}

\begin{definition}
A statistical manifold $(M, g, \nabla)$ is said to be {\sl conjugate symmetric}\/
if $R=L$ holds. 
\end{definition}

\begin{remark}
The conjugate symmetricity of $(M,g,\nabla)$ is equivalent to each of the following conditions.
\renewcommand{\labelenumi}{$(\theenumi)$}
\begin{enumerate}
    \item $R=\overline{R}$
    \item $\nabla^gK$ is symmetric on $M$.
\end{enumerate}
\end{remark}

\vspace{0.5\baselineskip}

\begin{definition}
    A statistical manifold $(M,g,\nabla)$ is said to be of constant curvature $\lambda$ if it satisfies
    \begin{equation*}
        R(X,Y)Z=\lambda(g(Y,Z)X-g(X,Z)Y),\quad X,Y,Z\in\Gamma(TM)
    \end{equation*}
    for a real number $\lambda$.
\end{definition}

\vspace{0.5\baselineskip}

\begin{definition}
Let $(M,g,\nabla)$ be a statistical manifold. The vector field $T=T^{(g, \nabla)}=\mathrm{tr}_gK \in \Gamma(TM)$ is called the {\sl Tchebychev vector field} of $(M,g,\nabla)$. Here, if $T=0$, then $(M, g, \nabla)$ is said 
to satisfy the {\sl equiaffine condition}\/ 
or to satisfy the {\sl apolarity condition}\/.  
\end{definition}

The next remark follows from a straightforward calculation (see \cite{NomizuSasaki} for details).
\begin{remark}
    \label{symric}
    The Tchebychev vector field on statistical manifolds gives the following properties.
    \renewcommand{\labelenumi}{$(\theenumi)$}
    \begin{enumerate}
        \item On a statistical manifold $(M,g,\nabla)$, the symmetry of $\operatorname{Ric}$ is equivalent to $T$ satisfying the following equation$:$
    \begin{equation}
    \label{symvec}
        g(\nabla^g_XT,Y)=g(\nabla^g_YT,X),\quad X,Y\in\Gamma(TM).
    \end{equation}
        \item If there exists a function $\varphi\in C^{\infty}(M)$ such that $g(T,X)=d\varphi(X)$ for all $X\in\Gamma(TM)$, then the volume form $\theta=e^{\varphi}\omega_g$ is $\nabla$-parallel, that is, $\nabla\theta=0$ holds. Here, $\omega_g$ is the volume form of $g$. 
        In particular, the condition $T=0$ on a statistical manifold is equivalent to the volume form $\omega_g$ of $g$ satisfying $\nabla \omega_g=0$.
    \end{enumerate}
\end{remark}

\vspace{0.5\baselineskip}

\begin{example}
\label{eqs}
    Let $(M,g)$ be a Riemannian manifold, and $C\in\Gamma(TM^{(0,3)})$ a totally symmetric tensor field. We define an affine connection $\nabla$ on $M$ by
    \begin{equation}
        \label{cubic}
        g(\nabla_X Y,Z)=g(\nabla^g_X Y,Z)-\frac{1}{2}C(X,Y,Z),\quad X,Y,Z\in\Gamma(TM). 
    \end{equation}
    Then, the pair $(g,\nabla)$ is a statistical structure on $M$. Here, we have 
    \renewcommand{\labelenumi}{$(\theenumi)$}
    \begin{enumerate}
        \item $C= \nabla g$,
        \item $C(X,Y,Z)=-2g(K_X Y,Z)$ for $X,Y,Z\in\Gamma(TM)$,
        \item $g(T,X)=-2\mathrm{tr}_gC(\cdot,\cdot,X)$ for $X\in\Gamma(TM)$.
    \end{enumerate}
    
\end{example}

\vspace{0.5\baselineskip}

\begin{example}\label{Centroaffine Geometry}
Let $\mathbb{R}^{m+1}$ be the real coordinate vector space, and $D$ be the standard flat affine connection on $\mathbb{R}^{m+1}$. 
The general linear group $GL(m+1, \mathbb{R})$ acts on this space, preserving $D$ invariant. 
The properties found on hypersurfaces in $\mathbb{R}^{m+1}$ that are invariant under the action induced from that of $GL(m+1, \mathbb{R})$, is the study of centroaffine geometry.
Let $f:M \to \mathbb{R}^{m+1}$ be a centroaffine immersion, that is, $f$ is an immersion from a manifold $M$ to $\mathbb{R}^{m+1}$ that satisfies
$T_{f(x)}\mathbb{R}^{m+1}=(df)_xT_xM \oplus \mathbb{R} f(x)$ at each point $x \in M$. 
The surface $M$ is called a centroaffine surface. 
According to this decomposition, define $\nabla, h$ by 
\begin{eqnarray*}
&&
D_Xf_*Y=f_*\nabla_XY- h(X,Y)f, \quad X,Y\in \Gamma(TM).
\end{eqnarray*}
The affine connection $\nabla$ is a torsion-free projectively flat affine connection, and $h\in\Gamma(TM^{(0,2)})$ is a symmetric tensor field that satisfies
$$
h=\frac{1}{m-1}\operatorname{Ric}^{\nabla}.
$$
If $h$ is a definite metric, the centroaffine immersion $f$ is called \textit{l.s.c.}, which is short for locally strongly convex. 
Define $g=\pm h$, where the sign is chosen such that $g$ is a positive definite metric.
Since $\nabla$ is projectively flat,
$$
\nabla g = \pm\frac{1}{m-1}\nabla(\operatorname{Ric}^{\nabla})
$$
is totally symmetric. 
Therefore, the pair $(g,\nabla)$ obtained here is a statistical structure on $M$, called the statistical structure induced by a l.s.c. centroaffine immersion. Here, we have 
\begin{equation}
    R(X,Y)Z = \pm(g(Y,Z)X-g(X,Z)Y),\quad X,Y,Z\in\Gamma(TM),
\end{equation}
thus the statistical manifold obtained by an l.s.c. centroaffine immersion is always of constant curvature $\pm 1$.
On the other hand, it is known that statistical manifolds of constant curvature $\pm 1$ can locally be induced by an l.s.c. centroaffine immersion \cite{MR1030710}. 
Additionally, if the Tchebychev vector field $T$ vanishes on $M$, the centroaffine hypersphere is called a proper affine hypersphere, an important subject in the field of equiaffine geometry.

The tensor field $\mathcal{T}=\mathcal{T}^{(g,\nabla)}=\nabla^gT=\in \Gamma(TM^{(1,1)})$ is called the {\sl Tchebychev operator}, it serves as the shape operator of centroaffine surfaces \cite{MR1286138}. 
By Theorem 3.1 in\cite{Furuhata2006TheCM}, the Tchebychev operator $\mathcal{T}$ on centroaffine hypersurface satisfies $\mathcal{T}=0$ if the center map of $f$ is centroaffinely congruent to $f$ itself. 
\end{example}

Here is one example of centroaffine immersion from \cite{Furuhata2006TheCM}.
\begin{example}
    Let $M=\{(x^1,x^2)\in\mathbb{R}^2 \mid x^1,x^2>0\}$ and define a locally strongly convex centroaffine immersion by $f:M\to\mathbb{R}^3;(x^1,x^2)\mapsto(x^1,x^2,(x^1)^{-a_1}(x^2)^{-a_2})$, with $a_1,a_2>0$. Induce a statistical structure $(g,\nabla)$ on $M$ by 
    \begin{equation*}
        D_Xf_*Y=f_*(\nabla_X Y)+g(X,Y)f,\quad X,Y\in\Gamma(TM).
    \end{equation*}
    We have
    \begin{equation*}
        g = \frac{1}{a_1+a_2+1}\sum_{i=1}^2\frac{a_i(a_j+\delta_{ij})}{x^ix^j}dx^i\otimes dx^j,
    \end{equation*}
    \begin{equation*}
        \nabla_{\frac{\partial}{\partial x^i}}\frac{\partial}{\partial x^j}=-\frac{a_i(a_j+\delta_{ij})}{a_i+a_j+1}\left(\frac{x^1}{x^ix^j}\frac{\partial}{\partial x^i}+\frac{x^2}{x^ix^j}\frac{\partial}{\partial x^j}\right),
    \end{equation*}
    \begin{equation*}
        \nabla^g_{\frac{\partial}{\partial x^i}}\frac{\partial}{\partial x^j}=-\frac{\delta_{ij}}{x^i}\frac{\partial}{\partial x^i}.
    \end{equation*}
    The Tchebychev vector field $T$ is given by
    \begin{equation*}
        g(T,\cdot)=(1-a_1)\frac{1}{x^1}dx^1+(1-a_2)\frac{1}{x^2}dx^2,
    \end{equation*}
    while the Tchebychev operator $\mathcal{T}$ vanishes globally on $M$.
\end{example}

\subsection{Connections on the vector bundles over statistical manifolds}
In this subsection, we use notations and definitions from \cite{FU}. 
Let $u:(M,g,\nabla^M)\to(N,h,\nabla^N)$ be a smooth mapping between statistical manifolds. 
On the vector bundle $u^{-1}TN$, we denote by $\nabla^u$ the connection induced from $\nabla^N$, that is, $\nabla^u_X U=\nabla^N_{u_* X}U \in \Gamma(u^{-1}TN)$ for 
$X \in \Gamma(TM)$ and $U \in \Gamma(u^{-1}TN)$.
We also denote by $\overline{\nabla}^{\substack{\scalebox{0.45}{\phantom{i}}\\u}}$, $\widehat{\nabla}^u$ for the connections on $u^{-1}TN$ induced by $\overline{\nabla}^{\substack{\scalebox{0.3}{\phantom{i}}\\N}}$, $\nabla^h$, respectively.

\vspace{0.5\baselineskip}

\begin{definition}[\cite{FU}]
\label{FUdef}
Let $(M, g, \nabla^M)$ be a statistical manifold, and $E \to M$ a vector bundle 
with connection $\nabla^E : \Gamma(E) \to \Gamma(E \otimes T^*M)$. 
We define the operator 
$\Delta^E=\Delta^{(g, \nabla^M, \nabla^E)} : \Gamma(E) \to \Gamma(E)$ as 
\begin{equation*}
    \Delta^E \xi 
    =\mathrm{tr}_g \left\{ (X,Y) \mapsto \nabla^E_X\nabla^E_Y \xi -\nabla^E_{\nabla^M_X Y}\xi \right\}, 
\end{equation*}
and call it the {\sl statistical connection Laplacian}\/ 
with respect to $(g, \nabla^M)$ and $\nabla^E$. 
\end{definition}

In Definition \ref{FUdef}, we denote $\Delta^E=\Delta^u$ if $E=u^{-1}TN$ and $\nabla^E=\nabla^u$. 
Similarly, we denote the statistical connection Laplacian by $\bar{\Delta}^u$, $\widehat{\Delta}^u$ if $E=u^{-1}TN$ and $\nabla^E=\overline{\nabla}^{\substack{\scalebox{0.45}{\phantom{i}}\\u}}$, $\nabla^E=\widehat{\nabla}^u$, respectively. 
If $(M,g,\nabla^M)$ is Riemannian, $E=M\times\mathbb{R}$, and $\nabla^E$ is the standard affine connection on $\mathbb{R}$, then $\Delta^E$ is the standard Laplace operator $\Delta_g$ on $(M,g)$. If $(M,g,\nabla^M)$ is Riemannian, $(M,g,\nabla^M)=(N,h,\nabla^N)$, and $u=\mathrm{id}$, then $\Delta^E$ is the rough Laplacian which we will also denote by $\Delta_g$. 
For a vector field $V\in\Gamma(TM)$ that satisfies equation $(\ref{symvec})$, the following equation holds on $(M,g)$ for $X\in\Gamma(TM)$:
\begin{equation}
    \label{intchangelap}X\operatorname{div}^g(V)=g(\Delta_gV,X)-\operatorname{Ric}^g(V,X).
\end{equation}
The proof can be found in \cite{10.2969/jmsj/03620295}.

\section{Statistical biharmonic maps}\label{sec3}
In this section, we will go over statistical biharmonic maps between statistical manifolds. 
For more details, see \cite{FU}. 
Let $(M,g,\nabla)$ be a statistical manifold, $(N,\nabla^N)$ a manifold equipped with a torsion-free affine connection, and let $u:M\to N$ a smooth map. 
The tension field $\tau(u)=\tau^{(g,\nabla,\nabla^N)}(u)$ of $u$ is defined by 
\begin{equation*}
\tau(u) = \mathrm{tr}_g\{(X,Y)\mapsto\nabla^u_X u_*Y-u_*\nabla_X Y\} 
        \in \Gamma(u^{-1}TN).
\end{equation*}
Here, $\nabla^u$ is the connection induced on $u^{-1}TN$ by $\nabla^N$.

\begin{example}
    Let $c:\left((a,b), g_0, \nabla^{g_0}\right)\to(N,\nabla^N)$ be a curve on $(N,\nabla^N)$, where $g_0$ is the Euclidean metric. 
    If we let $\dot{c}=c_*\dfrac{d}{dt}$, then $\tau(c)=\nabla^N_{\dot{c}}\dot{c}$ holds.
\end{example}

\vspace{0.5\baselineskip}

\begin{example}
\label{id}
    Let $(M,g,\nabla)$ be a statistical manifold. 
    Consider $\mathrm{id}:(M,g,\nabla)\to(N,\nabla^g)$, the identity map. Then, we have
    \begin{equation*}
        \tau(\mathrm{id})=-T.
    \end{equation*}

\end{example}

Let $(M,g,\nabla^M)$ and $(N,h,\nabla^N)$ be statistical manifolds and $\overline{\nabla}^{\substack{\scalebox{0.3}{\phantom{i}}\\M}}$ and $ \overline{\nabla}^{\substack{\scalebox{0.3}{\phantom{i}}\\N}}$ their conjugate connections, respectively. 
For a mapping $u:M\to N$, denote by $\tau(u)=\tau^{(g,\nabla^M,\nabla^N)}(u)$ and $\overline{\tau}(u)=\tau^{(g,\overline{\nabla}^{\substack{\scalebox{0.01}{\phantom{i}}\\\scalebox{0.5}{$M$}}},\overline{\nabla}^{\substack{\scalebox{0.01}{\phantom{i}}\\\scalebox{0.5}{$N$}}})}(u)$. 
We also denote by $\widehat{\tau}(u)=\tau^{(g,\nabla^g,\nabla^h)}(u)$, which is the usual tension field for mappings between Riemannian manifolds. 
A mapping $u:(M,g)\to(N,h)$ satisfying $\widehat{\tau}(u)=0$ is called a \textit{harmonic map}. 
For a mapping $u:(M,g,\nabla^M)\to(N,\nabla^N)$, remark that the equation $\tau(u)=0$ only defines a class of mapping between $(M,g,\nabla^M)$ and $(N,\nabla^N)$, not between two statistical manifolds.

\begin{example}
\label{id2}
    Consider the identity map $\mathrm{id}:(M,g,\nabla)\to(M,\nabla^g)$ in Example \ref{id}. 
    The conjugate connection of $(M,g,\nabla^g)$ is $\nabla^g$, thus the relation $\overline{\tau}(\mathrm{id})=\tau^{(g,\overline{\nabla},\nabla^g)}(\mathrm{id})=T=-\tau(\mathrm{id})$ holds.
    It is clear that $\widehat{\tau}(\mathrm{id})=\tau^{(g,\nabla^g,\nabla^g)}(\mathrm{id})=0$.
\end{example}

The following proposition holds from the definition of $\tau(u)$, $\overline{\tau}(u)$, $\widehat{\tau}(u)$, and equation (\ref{Leviconj}).

\begin{proposition}
\label{difftension}
    Let $(M,g,\nabla^M), (N,h,\nabla^N)$ be statistical manifolds. 
    For a mapping $u:M \to N$, we have
    \begin{equation*}
\widehat{\tau}(u)=\frac{\tau(u)+\overline{\tau}(u)}{2}.
    \end{equation*}
\end{proposition}

\vspace{0.5\baselineskip}

Now, between statistical manifolds $(M,g,\nabla^M)$ and $(N,h,\nabla^N)$, we will derive a class of mappings from a variational principle. 

\begin{definition}
Let $M$ be compact. The statistical bi-energy of $u:(M,g,\nabla^M)\to (N,h,\nabla^N)$ is defined by
\begin{eqnarray*}
    E_2(u)
    &=&E_2^{(g, \nabla^M, h, \nabla^N)}(u)  \\
    &=&\frac{1}{2} \int_{M} \| \tau(u) \|_{h}^2 d\mu_g.  \nonumber
\end{eqnarray*}
Here, $\|\cdot\|_h$ is the norm derived by $h$ and $d\mu_g$ is the standard measure derived from the Riemanian metric $g$. 
\end{definition}

The statistical bi-energy functional is defined by using the statistical structures of both the source and the target.
The statistical bi-tension field $\tau_2(u)=\tau_2^{(g,\nabla^M, h,\nabla^N)}(u)$ of $u$ is defined by
\begin{eqnarray}
\tau_2(u)
&=&        
\bar{\Delta}^u\tau(u)+\operatorname{div}^g
(\mathrm{tr}_g K^M)\tau(u) \nonumber\\
&&\quad 
-\sum_{i=1}^mL^N(u_*e_i,\tau(u)) u_*e_i -K^N(\tau(u),\tau(u) )
\quad 
\in \Gamma(u^{-1}TN),      \nonumber
\end{eqnarray}
where $\{e_i\}_{i=1, \ldots,m}$ is an orthonormal frame with respect to $g$.

The next theorem is proved in \cite{FU}.
\begin{theorem}
For any variations $F=\{u_t\}_{t\in(-\epsilon,\epsilon)}$ of $u$, we have
\begin{equation*}
  \left.\displaystyle\frac{d}{dt}\right|_{t=0}E_2(u_t)
    =\int_{M}\left\langle V,\tau_2(u)\right\rangle d\mu_g.
\end{equation*}
Here, $V\in\Gamma(u^{-1}TN)$ is the variational vector field generated by $F$.
\end{theorem}

\vspace{0.5\baselineskip}

Let $(M,g,\nabla^M), (N,h,\nabla^N)$ be statistical manifolds and $\overline{\nabla}^{\substack{\scalebox{0.3}{\phantom{i}}\\M}}$, $ \overline{\nabla}^{\substack{\scalebox{0.3}{\phantom{i}}\\N}}$ their conjugate connections, respectively. For a mapping $u: M\to N$, we denote $\overline{\tau}_2(u)=\tau^{(g,\overline{\nabla}^{\substack{\scalebox{0.01}{\phantom{i}}\\ \scalebox{0.55}{$M$}}},h,\overline{\nabla}^{\substack{\scalebox{0.01}{\phantom{i}}\\ \scalebox{0.55}{$N$}}})}_2(u)$.

\begin{proposition}
\label{diffbitensionpr}
     Let $u: (M,g, \nabla^M) \to (N, h, \nabla^N)$ be a smooth map between statistical manifolds. Then, we have 
     \begin{equation}
     \label{diffbitension}
     \begin{split}
         \tau_2(u)-\overline{\tau}_2(u)&=2\left(\widehat{\Delta}^u\tau(u)-\sum_{i=1}^mR^h(u_*e_i,\tau(u))u_*e_i\right)\\
         &+2K^N(\tau(u),\widehat{\tau}(u))-4K^N(\widehat{\tau}(u),\widehat{\tau}(u))\\
         &-2\left(\Delta^u\widehat{\tau}(u)-\operatorname{div}^g(\mathrm{tr}_gK^M)\widehat{\tau}(u)-\sum_{i=1}^m\overline{L}^N(u_*e_i,\widehat{\tau}(u))u_*e_i\right).
     \end{split}
     \end{equation}
     In particular, if $u:(M,g)\to(N,h)$ is a harmonic map, that is, $\widehat{\tau}(u)=0$, then we have
     \begin{equation*}
         \tau_2(u)-\overline{\tau}_2(u)=2\left(\widehat{\Delta}^u\tau(u)-\sum_{i=1}^mR^h(u_*e_i,\tau(u))u_*e_i\right).
     \end{equation*}
\end{proposition}
\begin{proof}
    The equation
    \begin{eqnarray}
\overline{\tau}_2(u)
&=&        
\Delta^u\overline{\tau}(u)-\operatorname{div}^g
(\mathrm{tr}_g K^M)\overline{\tau}(u) \nonumber\\
&&\quad 
-\sum_{i=1}^m\overline{L}^N(u_*e_i,\overline{\tau}(u)) u_*e_i +K^N(\overline{\tau}(u),\overline{\tau}(u) ) \nonumber
\end{eqnarray}
holds by definition.
    By equation (\ref{sumcurv}) and Proposition \ref{difftension}, we have
    \begin{equation}
    \label{diffbiteneq1}
        \begin{split}
            \tau_2(u)-\overline{\tau}_2(u)&=\Delta^u\tau(u)+\bar{\Delta}^u\tau(u)-\sum_{i=1}^m(R^N+\overline{R}^{\substack{\scalebox{0.3}{\phantom{i}}\\N}})(u_*e_i,\tau(u))u_*e_i\\
            &-2K^N(\tau(u),\tau(u))+4K^N(\tau(u),\widehat{\tau}(u))-4K^N(\widehat{\tau}(u),\widehat{\tau}(u))\\
            &-2\left(\Delta^u\widehat{\tau}(u)-\operatorname{div}^g(\mathrm{tr}_gK^M)\widehat{\tau}(u)-\sum_{i=1}^m\overline{L}^N(u_*e_i,\widehat{\tau}(u))u_*e_i\right).
        \end{split}
    \end{equation}
    Here, since the affine connections are torsion-free, we have 
    
    \begin{equation*}
        \begin{split}
            \Delta^u\tau(u)+\bar{\Delta}^u\tau(u)&-\sum_{i=1}^m(R^N+\overline{R}^{\substack{\scalebox{0.3}{\phantom{i}}\\N}})(u_*e_i,\tau(u))u_*e_i\\
            &=\sum_{i=1}^m\left(\nabla^u_{e_i}\nabla^u_{e_i}-\nabla^u_{\substack{ \scalebox{0.1}{\phantom{i}}\\\nabla^M_{e_i}e_i}}\right)\tau(u)+\sum_{i=1}^m\left(\overline{\nabla}^{\substack{\scalebox{0.45}{\phantom{i}}\\u}}_{e_i}\overline{\nabla}^{\substack{\scalebox{0.45}{\phantom{i}}\\u}}_{e_i}-\overline{\nabla}^{\substack{\scalebox{0.45}{\phantom{i}}\\u}}_{\overline{\nabla}^{\substack{\scalebox{0.1}{\phantom{i}}\\ \scalebox{0.55}{$M$}}}_{e_i}e_i}\right)\tau(u)\\
            &-\sum_{i=1}^m\left(\nabla^N_{u_*e_i}\nabla^N_{\tau(u)}u_*e_i-\nabla^N_{\tau(u)}\nabla^N_{u_*e_i}u_*e_i-\nabla^N_{\lbrack u_*e_i,\tau(u)\rbrack}u_*e_i\right)\\
            &-\sum_{i=1}^m\left(\overline{\nabla}^{\substack{\scalebox{0.3}{\phantom{i}}\\N}}_{u_*e_i}\overline{\nabla}^{\substack{\scalebox{0.3}{\phantom{i}}\\N}}_{\tau(u)}u_*e_i-\overline{\nabla}^{\substack{\scalebox{0.3}{\phantom{i}}\\N}}_{\tau(u)}\overline{\nabla}^{\substack{\scalebox{0.3}{\phantom{i}}\\N}}_{u_*e_i}u_*e_i-\overline{\nabla}^{\substack{\scalebox{0.3}{\phantom{i}}\\N}}_{\lbrack u_*e_i,\tau(u)\rbrack}u_*e_i\right)\\
            &=\sum_{i=1}^m\left(\nabla^N_{u_*e_i}\lbrack u_*e_i,\tau(u)\rbrack+\nabla^N_{\lbrack u_*e_i,\tau(u)\rbrack}u_*e_i\right)\\
            &+\sum_{i=1}^m\left(\overline{\nabla}^{\substack{\scalebox{0.3}{\phantom{i}}\\N}}_{u_*e_i}\lbrack u_*e_i,\tau(u)\rbrack+\overline{\nabla}^{\substack{\scalebox{0.3}{\phantom{i}}\\N}}_{\lbrack u_*e_i,\tau(u)\rbrack}u_*e_i\right)\\
            &+\sum_{i=1}^m\left(\nabla^N_{\tau(u)}\nabla^N_{u_*e_i}u_*e_i-\nabla^u_{\substack{ \scalebox{0.1}{\phantom{i}}\\\nabla^M_{e_i}e_i}}\tau(u)+\overline{\nabla}^{\substack{\scalebox{0.3}{\phantom{i}}\\N}}_{\tau(u)}\overline{\nabla}^{\substack{\scalebox{0.3}{\phantom{i}}\\N}}_{u_*e_i}u_*e_i-\overline{\nabla}^{\substack{\scalebox{0.45}{\phantom{i}}\\u}}_{\overline{\nabla}^{\substack{\scalebox{0.1}{\phantom{i}}\\ \scalebox{0.55}{$M$}}}_{e_i}e_i}\tau(u)\right).
        \end{split}
    \end{equation*}
    In the last summation, from the definition of $\tau(u)$,
    \begin{equation*}
    \begin{split}
        \sum_{i=1}^m\left(\nabla^N_{\tau(u)}\nabla^N_{u_*e_i}u_*e_i-\nabla^u_{\substack{ \scalebox{0.1}{\phantom{i}}\\\nabla^M_{e_i}e_i}}\tau(u)\right)&=\nabla^N_{\tau(u)}\tau(u)+\sum_{i=1}^m\left\lbrack \tau(u), u_*\left(\nabla^M_{e_i}e_i\right)\right\rbrack
    \end{split}
    \end{equation*}
    holds. By equations $(\ref{Kconjg})$ and $(\ref{Leviconj})$, we have
    \begin{equation*}
        \begin{split}
            \Delta^u\tau(u)+\bar{\Delta}^u\tau(u)&-\sum_{i=1}^m(R^N+\overline{R}^{\substack{\scalebox{0.3}{\phantom{i}}\\N}})(u_*e_i,\tau(u))u_*e_i\\
            &=2\sum_{i=1}^m\left(\nabla^h_{u_*e_i}\lbrack u_*e_i,\tau(u)\rbrack+\nabla^h_{\lbrack u_*e_i,\tau(u)\rbrack}u_*e_i\right)\\
            &+\nabla^N_{\tau(u)}\tau(u)+\sum_{i=1}^m\left(\lbrack \tau(u), u_*(\nabla^M_{e_i}e_i)\rbrack\right)\\
            &+\overline{\nabla}^{\substack{\scalebox{0.3}{\phantom{i}}\\N}}_{\tau(u)}\overline{\tau}(u)+\sum_{i=1}^m(\lbrack \tau(u), u_*(\overline{\nabla}^{\substack{\scalebox{0.3}{\phantom{i}}\\M}}_{e_i}e_i)\rbrack)\\
            &=2\sum_{i=1}^m\left(\nabla^h_{u_*e_i}\lbrack u_*e_i,\tau(u)\rbrack+\nabla^h_{\lbrack u_*e_i,\tau(u)\rbrack}u_*e_i+\left\lbrack \tau(u), u_*\left(\nabla^g_{e_i}e_i\right)\right\rbrack\right)\\
            &+2K^N(\tau(u),\tau(u))+2\overline{\nabla}^{\substack{\scalebox{0.3}{\phantom{i}}\\N}}_{\tau(u)}\widehat{\tau}(u).
        \end{split}
    \end{equation*}
    In the summation, since $\nabla^g,\nabla^h$ are torsion-free, we have
    \begin{equation*}
        \begin{split}
            &2\sum_{i=1}^m\left(\nabla^h_{u_*e_i}\lbrack u_*e_i,\tau(u)\rbrack+\nabla^h_{\lbrack u_*e_i,\tau(u)\rbrack}u_*e_i+\left\lbrack \tau(u), u_*\left(\nabla^g_{e_i}e_i\right)\right\rbrack\right)\\
            &=2\sum_{i=1}^m\left(\nabla^h_{u_*e_i}\nabla^h_{u_*e_i}\tau(u)-\nabla^h_{u_*e_i}\nabla^h_{\tau(u)}u_*e_i+\nabla^h_{\lbrack u_*e_i,\tau(u)\rbrack}u_*e_i\right)\\
            &+2\sum_{i=1}^m\left(\nabla^h_{\tau(u)}u_*(\nabla^g_{e_i}e_i)-\nabla^h_{u_*(\nabla^g_{e_i}e_i)}\tau(u)\right)\\
            &=2\widehat{\Delta}^u\tau(u)-2\sum_{i=1}^mR^h\left(u_*e_i,\tau(u)\right)u_*e_i-2\nabla^h_{\tau(u)}\widehat{\tau}(u),
        \end{split}
    \end{equation*}
    thus
    \begin{equation}
    \label{lasteq}
        \begin{split}       \Delta^u\tau&(u)+\bar{\Delta}^u\tau(u)-\sum_{i=1}^m(R^N+\overline{R}^{\substack{\scalebox{0.3}{\phantom{i}}\\N}})(u_*e_i,\tau(u))u_*e_i\\
        &=2\widehat{\Delta}^u\tau(u)-2\sum_{i=1}^mR^h\left(u_*e_i,\tau(u)\right)u_*e_i+2K^N(\tau(u),\tau(u))-2K^N(\tau(u),\widehat{\tau}(u)).
        \end{split}
    \end{equation}
    By substituting (\ref{lasteq}) in equation (\ref{diffbiteneq1}), equation (\ref{diffbitension}) follows.
    
\end{proof}

\section{Statistical manifolds with statistical biharmonic identity maps}
\subsection{Statistical biharmonicity of identity maps}
As we have seen in Examples \ref{id} and \ref{id2}, for a statistical manifold $(M,g,\nabla)$ the identity map $\mathrm{id}:M\to M$ satisfies $\tau(\mathrm{id})=\tau^{(g,\nabla,\nabla^g)}(\mathrm{id})=0$ and $\overline{\tau}(\mathrm{id})=\tau^{(g,\overline{\nabla},\nabla^g)}=0$ if and only if the statistical manifold $(M,g,\nabla)$ satisfies the equiaffine condition.
\begin{theorem}
\label{main1}
    Let $(M,g,\nabla)$ be a statistical manifold, $\mathrm{id}:(M,g,\nabla)\to (M,g,\nabla^g)$ the identity map. 
    Denote $\tau_2(\mathrm{id})=\tau_2^{(g,\nabla,g,\nabla^g)}(\mathrm{id})$ and $\overline{\tau}_2(\mathrm{id})=\tau_2^{(g,\overline{\nabla},g,\nabla^g)}(\mathrm{id})$. 
    The equation $\tau_2(\mathrm{id})=\overline{\tau}_2(\mathrm{id})=0$ holds if and only if the equations $\operatorname{(T1)}$ and $\operatorname{(T2)}$ hold for the Tchebychev vector field.
\end{theorem}
\begin{proof}
    The equation $\tau_2(\mathrm{id})=\overline{\tau}_2(\mathrm{id})=0$ is equivalent to the following two equations to hold$:$\\
    \centerline{1. \ 
    $\tau_2(\mathrm{id})-\overline{\tau}_2(\mathrm{id})=0$,}

    \centerline{2. \
    $\tau_2(\mathrm{id})+\overline{\tau}_2(\mathrm{id})=0$.}
Here, equation $1.$ is equivalent to equation (\ref{eq1}) by Proposition \ref{diffbitensionpr} and the definition of $\operatorname{Ric}^g$, since $\mathrm{id}:(M,g)\to(M,g)$ is a harmonic map. 

As for equation $2.$, since 
\begin{equation*}
\begin{split}
    \tau_2(\mathrm{id})&=-\sum_{i=1}^m\left(\nabla^g_{e_i}\nabla^g_{e_i}-\nabla^g_{\overline{\nabla}_{e_i}e_i}\right)T-\operatorname{div}^g(T)T+\sum_{i=1}^mR^g(e_i,T)e_i,\\
    \overline{\tau}_2(\mathrm{id})&=\sum_{i=1}^m\left(\nabla^g_{e_i}\nabla^g_{e_i}-\nabla^g_{\nabla_{e_i}e_i}\right)T-\operatorname{div}^g(T)T-\sum_{i=1}^mR^g(e_i,T)e_i,
\end{split}
\end{equation*}
and from the definition of $T$, we have
\begin{equation*}
    -\frac{1}{2}(\tau_2(\mathrm{id})+\overline{\tau}_2(\mathrm{id}))=\operatorname{div}^g(T)T+\nabla^g_T T.
\end{equation*}
\end{proof}

\vspace{0.5\baselineskip}

\begin{definition}
    Let $(M,g,\nabla)$ be a statistical manifold. 
    If the equations (\ref{eq1}) and (\ref{eq2}) are satisfied on $M$, then $(M,g,\nabla)$ is said to \textit{satisfy the semi-equiaffine condition}.
\end{definition}

\begin{remark}
    Generally, a vector field $V$ satisfying the following equation on a Riemannian manifold $(M,g)$ is called a \textit{geodesic vector field}, in the sense of K. Yano \cite{MR124854}$:$
    \begin{equation}
        \label{geovec1}
        \Delta_g V + \sum_{i=1}^m\mathrm{Ric}^g(e_i,V)e_i=0.
    \end{equation}
\end{remark}

\begin{remark}
    A vector field $V$ on a Riemannian manifold $(M,g)$ is also called a \textit{geodesic vector field} if it satifies the following equation, but in the sense of S. Deshmukh, P. Peska and N. Bin Turki \cite{DPN}.
    \begin{equation}
        \label{geovec2}
        \nabla^g_V V = \rho V.
    \end{equation}
    Here, $\rho$ is a smooth function on $M$, called the potential of $V$.
    In this sense, geodesic vector fields are vector fields whose integral curves are reparameterized geodesics.
\end{remark}

\subsection{The semi-equiaffine condition}
We observe examples and properties of statistical manifolds that satisfy the semi-equiaffine condition.
\begin{proposition}
    Let $(M,g,\nabla)$ be a statistical manifold with symmetric $\operatorname{Ric}$ and $g$ is an Einstein metric of non-zero Einstein constant, that is, $\operatorname{Ric}^g=c\cdot g$ holds with $c\neq0$. 
    Assume that the equation $\operatorname{(T1)}$ is satisfied on $M$. Then, there exists a $\nabla$-parallel volume form $\theta$ on $M$. 
\end{proposition}
\begin{proof}
From equation $(\ref{intchangelap})$ and that $\operatorname{Ric}$ is symmetric, we have
$$
X\operatorname{div}^g(T)=g(\Delta_gT,X)-\operatorname{Ric}^g(T,X),\quad X\in\Gamma(TM).
$$
Here, from equation (\ref{eq1}) and that $\operatorname{Ric}^g=c\cdot g$, we obtain
\begin{equation}
    \label{einmet}
    X\operatorname{div}^g(T)=-2c\cdot g(T,X),\quad X\in\Gamma(TM).
\end{equation}
Therefore, if we let $\varphi=-\frac{1}{2c}\operatorname{div}^g(T)$, we have 
\begin{equation*}
    g(T,X)=d\varphi (X),\quad X\in\Gamma(TM).
\end{equation*}
By Remark \ref{symric}, $\theta=e^{\varphi}\cdot\omega_g$ is a $\nabla$-parallel volume form.
\end{proof}

\vspace{0.5\baselineskip}

\begin{proposition}
    \label{parT}
    Let $(M,g,\nabla)$ be a statistical manifold.
    \renewcommand{\labelenumi}{$(\theenumi)$}
    \begin{enumerate}
        \item If $\nabla^gT=0$ on $M$, then $(M,g,\nabla)$ satisfies the semi-equiaffine condition. \label{parT1}
        \item Assume that $(M,g,\nabla)$ satisfies the semi-equiaffine condition. \label{parT2}
        If $\operatorname{Ric}$ is symmetric and $\operatorname{Ric}^g(T,T)\leq0$, then we have $\nabla^gT=0$ on $M$.
    \end{enumerate}
\end{proposition}
\begin{proof}
    For $(\ref{parT1})$, The equations (\ref{eq1}) and (\ref{eq2}) both obviously hold if $\nabla^gT=0$ on $M$.\\
    We prove $(\ref{parT2})$ next.
    If $\operatorname{Ric}$ is symmetric, by Remark \ref{symric} for any $X\in\Gamma(TM)$ we have
    \begin{equation*}
        g(\nabla^g_T T,X) = \frac{1}{2}Xg(T,T).
    \end{equation*}
    Hence, the equation (\ref{eq2}) can be written as
    \begin{equation}
    \label{eq2_2}
        g(\operatorname{div}^g(T)T,X) = -\frac{1}{2}Xg(T,T)
    \end{equation}
    for any $X\in\Gamma(TM)$. 
    Fix a point $p\in M$, and take an orthonormal frame $\{e_1,\dots,e_m\}$ around $p$ such that $(\nabla^{g}e_1)_p=\cdot\cdot\cdot=(\nabla^{g}e_m)_p=0$. 
    From equation (\ref{eq2_2}), we have
    \begin{equation}
        \label{eq2_3}\sum_{i=1}^me_ig(\operatorname{div}^g(T)T,e_i)=-\frac{1}{2}\sum_{i=1}^me_ie_ig(T,T).
    \end{equation}
    Here, on the left-hand side, from equations (\ref{intchangelap}) and (\ref{eq1}), we have
    \begin{equation}
        \label{eq2_4}
        \begin{split}
            \sum_{i=1}^me_ig(\operatorname{div}^g(T)T,e_i)&=\sum_{i=1}^mg(\nabla^g_{e_i}(\operatorname{div}^g(T)T),e_i)\\
            &=T\operatorname{div}^g(T)+(\operatorname{div}^g(T))^2\\
            &=g(\Delta_gT,T)-\operatorname{Ric}^g(T,T)+ (\operatorname{div}^g(T))^2\\
            &=-2\operatorname{Ric}^g(T,T)+ (\operatorname{div}^g(T))^2.           
        \end{split}
    \end{equation}
    The right-hand side of equation (\ref{eq2_3}) is equal to $-\frac{1}{2}\Delta_gg(T,T)$, and from equation (\ref{eq1}), we have
    \begin{equation*}
        \begin{split}
   \Delta_gg(T,T)&=2g(\Delta_gT,T)+2g(\nabla^gT,\nabla^gT)\\
            &=-2\operatorname{Ric}^g(T,T)+2g(\nabla^gT,\nabla^gT).
        \end{split}
    \end{equation*}
    Therefore, since the choice of $p\in M$ was arbitrary, we have
    \begin{equation*}
        -3\operatorname{Ric}^g(T,T)+(\operatorname{div}^g(T))^2+g(\nabla^gT,\nabla^gT)=0
    \end{equation*}
    on $M$ from equation (\ref{eq2_3}). 
    Since $\operatorname{Ric}^g\leq0$ on $M$, we have $\nabla^gT=0$.
\end{proof}

\vspace{0.5\baselineskip}

\begin{example}
\label{Tchopzero}
    Let $(g,\nabla)$ be a statistical structure on $M$ induced from a l.s.c. centroaffine immersion as in Example \ref{Centroaffine Geometry}. 
    If the Tchebychev operator vanishes on $M$, then $(M,g,\nabla)$ satisfies the semi-equiaffine condition. 
    Such classes of centroaffine surfaces are classified in \cite{Liu1995TheCT}.
\end{example}

\begin{example}
    Let $(\mathbb{R}^m,g_0)$ be the Eucidean space, $C\in\Gamma(TM^{(0,3)})$ a totally symmetric tensor field such that $\nabla^{g_0}C=0$. 
    Then, we obtain a statistical manifold $(\mathbb{R}^m,g_0,\nabla)$ by equation (\ref{cubic}). 
    Here, the Tchebychev vector field $T$ satisfies $\nabla^{g_0}T=0$. 
    The statistical structure can be projected on the standard Torus $T^m=\mathbb{R}^m/\mathbb{Z}^m$, thus we have a compact statistical manifold that satisfies the semi-equiaffine condition.
\end{example}

\section{Statistical manifolds of constant curvature metric}
In this section, we determine the statistical structure of a statistical manifold that satisfies the semi-equiaffine condition. 
The statistical manifold $(M,g,\nabla)$ is assumed to have constant curvature and for $(M,g)$ to be a complete Riemannian manifold of constant sectional curvature.

It is well known that for a Riemannian manifold $(M,g)$, there exists a Riemannian universal covering $\varpi:(\Tilde{M},\Tilde{g})\to (M,g)$. The Riemannian manifold $(\Tilde{M},\Tilde{g})$ is complete if and only if $(M,g)$ is complete. 
For a statistical manifold $(M,g,\nabla)$ and a universal covering $\varpi:\Tilde{M}\to M$, we can induce a statistical structure $(\Tilde{g},\Tilde{\nabla})$ on $\Tilde{M}$ by $\Tilde{g}=\varpi^*g$ and
\begin{equation}
\label{indcon}
    \varpi_*(\Tilde{\nabla}_X Y)=\nabla_{\varpi_*X}\varpi_*Y,\quad X,Y\in\Gamma(TM). 
\end{equation}
For the Tchebychev vector field $T$, $\Tilde{T}$ of $(M,g,\nabla)$, $(\Tilde{M},\Tilde{g},\Tilde{\nabla})$, respectively, the following equation holds:
\begin{equation*}
    \varpi_*(\Tilde{T})=T.
\end{equation*}
Therefore, $(\Tilde{M},\Tilde{g},\Tilde{\nabla})$ satisfies the equiaffine condition if and only if $(M,g,\nabla)$ satisfies the equiaffine condition. This equivalence also holds for the semi-equiaffine condition.

The following well known proposition classifies complete Riemannian manifolds of constant sectional curvature\cite{Hopf1926, Killing1891}.
\begin{proposition}
\label{Riemcsc}
    Let $(M^m,g)$ be a complete Riemannian manifold of constant sectional curvature $c$.
    Then, the Riemannian universal covering is isomorphic to either
    \renewcommand{\labelenumi}{$(\theenumi)$}
    \begin{enumerate}
        \item The Euclidean sphere $(S^m(c),g_0)$ of radius $c^{-\frac{1}{2}}$, if $c>0$.
        \item The Euclidean space $(\mathbb{R}^m,g_0)$, if $c=0$.
        \item The hyperbolic space $(\mathbb{H}^m(c),g_0)$ of constant sectional curvature $c$, if $c<0$.
    \end{enumerate}
\end{proposition}

\vspace{0.5\baselineskip}

To determine a statistical structure that uses a complete Riemannian metric of constant sectional curvature, we only need to consider the cases $(1)$, $(2)$, and $(3)$ in Proposition \ref{Riemcsc}. 
The next proposition will be used in the next subsection. For proof, see \cite{Berger1971LeSD} for example.

\begin{proposition}
\label{eigval}
    Let $\Delta_{g_0}$ be the Laplace operator of the Euclidean sphere $(S^m(c),g_0)$ of radius $c^{-\frac{1}{2}}$. 
    The eigenvalues $0=\lambda_0>\lambda_1>\cdot\cdot\cdot>\lambda_k\to-\infty$ of $\Delta_{g_0}$ are given by
    \begin{equation*}
        \lambda_k = -c\cdot k(k+m-1),\quad k=0,1,\ldots .
    \end{equation*}
\end{proposition}

\subsection{Determination of the Tchebychev vector field}
\quad On a statistical manifold $(M^m,g,\nabla)$ with a constant curvature metric $g$ of constant $c$, the equation (\ref{eq1}) is reduced to
\begin{equation*}
    \Delta_gT+c(m-1)T=0
\end{equation*}
for the Tchebychev vector field $T$.
The next lemma is Theorem 5.3 in \cite{MR124854}.

\begin{lemma}
\label{geov}
    Let $(M,g)$ be a compact Riemannian manifold with positive constant sectional curvature $c$. 
    For any vector field $V$ on $(M,g)$ satisfying equation $(\ref{geovec1})$, there exists a unique function $f\in C^{\infty}(M)$ and a unique Killing vector field $\xi\in\Gamma(TM)$ such that
    \begin{equation}
    \label{yano1}
        V=\operatorname{grad}_gf+\xi,
    \end{equation}
    \begin{equation}
    \label{yano2}
        \Delta_g f=-2c(m-1)f.
    \end{equation}
\end{lemma}

The next lemma is from Corollary 2.9 in \cite{10.2969/jmsj/03620295}. 
\begin{lemma}
\label{geov2}
    Let $V$ be a vector field on the Euclidean sphere $(S^2(c),g_0)$ satisfying the equation $\mathrm{(T1)}$. 
    Then, there exists a unique function $\varphi\in C^{\infty}(S^2(c))$ such that $L_Vg_0 = \varphi g_0$.
    Here, $L_Vg_0$ is the Lie derivative of $g_0$ with respect to $V$, that is,
    \begin{equation*}
        (L_Vg_0)(X,Y) = g_0(\nabla^{g_0}_XV,Y)+g_0(\nabla^{g_0}_YV, X),\quad X,Y\in\Gamma(TM).
    \end{equation*}
\end{lemma}

\vspace{0.5\baselineskip}

We can now determine the Tchebychev vector field of statistical structures on complete Riemannian manifolds with an Einstein metric.
\begin{theorem}
    \label{determtche}
        Let $(M,g,\nabla)$ be a statistical manifold satisfying the semi-equiaffine condition with symmetric Ricci tensor field. 
        Assume that $(M,g)$ is a complete simply-connected Riemannian manifold of constant curvature $c$.
        Then, the Tchebychev vector field $T$ is determined by the following$:$    \renewcommand{\labelenumi}{$\operatorname{\theenumi.}$}
        \renewcommand{\labelenumi}{$(\theenumi)$}
        \begin{enumerate}
            \item If $(M,g)=(S^m(c),g_0)$, then $T=0$.
            \item If $(M,g)$ is the Euclidean space $(\mathbb{R}^m,g_0)$, then there exists $a^1,\dots,a^m\in\mathbb{R}$ such that
            \begin{equation}
            \label{detT}
                T  = \sum_{i=1}^m a^i\frac{\partial}{\partial x^i}.
            \end{equation}
            Here, $\{x^1,\dots,x^m\}$ is a coordinate system on $\mathbb{R}^m$ such that $\nabla^{g_0}dx^i=0$.
            \item If $(M,g)=(\mathbb{H}^m(c),g_0)$, then $T=0$.
        \end{enumerate}
    \end{theorem}
\begin{proof}
    The case $(2)$ and $(3)$ follows from Proposition \ref{parT} $(2)$, since we obtain $\nabla^gT=0$ from $c\leq0$.
    Assume that $(M,g)=(S^m(c),g_0)$.
    From Lemma \ref{geov}, there exists a unique $f\in C^{\infty}(S^m(c))$ and a unique Killing vector field $\xi\in\Gamma(TS^m(c))$ such that the equations (\ref{yano1}) and (\ref{yano2}) hold for $T$. 
    As stated in Remark \ref{symric}, we obtain $\xi=0$ since $\operatorname{Ric}$ is symmetric. 
    From Proposition \ref{eigval}, we see that $-2c(m-1)$ is not an eigenvalue of $\Delta_{g_0}$ unless $m=2$. 
    Therefore, we have $f=0$ and $T=\operatorname{grad}_{g_0}f=0$ if $m\geq3$. 
    For $m=2$, from Lemma \ref{geov2} there exists a unique function $\varphi\in C^{\infty}(S^2(c))$ that satisfies $L_Tg_0 = \varphi g_0$. 
    Here, the following equality holds$:$
    \begin{equation*}
        \varphi = \operatorname{div}^{g_0}(T)=\Delta_{g_0}f.
    \end{equation*}
    From equation (\ref{eq2}), we have
    \begin{equation*}
        \begin{split}
            0&=g_0(\nabla^{g_0}_TT,T)+\operatorname{div}^{g_0}(T)g_0(T,T)\\
            &=\frac{1}{2}\varphi g_0(T,T)+\sum_{i=1}^2g_0(\nabla^{g_0}_{e_i}T,e_i)g_0(T,T)\\
            &=\frac{2}{3}\varphi g_0(T,T).
        \end{split}
    \end{equation*}
    Therefore, it is induced $\varphi=0$ or $g_0(T,T)=0$, and thus $T=0$.
\end{proof}
\subsection{Determination of the statistical structure}
We will determine the statistical structure $(g,\nabla)$ in the case where $(M^m,g,\nabla)$ is of constant curvature.

\begin{lemma}
\label{lapC}
    Let $(M,g,\nabla)$ be a conjugate symmetric statistical manifold. 
    If $\nabla^{g}T=0$ on $M$, the following equality holds for the difference tensor $K$$:$
    \begin{equation}
        \Delta_g g(K,K) = 2g(F,K) + 2g(\nabla^g K,\nabla^g K).
    \end{equation}
    Here, the tensor field $F\in\Gamma(T^{(1,2)}M)$ is defined by 
    \begin{equation*}
        F(X,Y)=\sum_{l=1}^m(R^g(e_l,X)K)(e_l,Y),\quad X,Y\in\Gamma(TM),
    \end{equation*}
    where $\{e_1,\dots,e_m\}$ is an orthonormal frame on $M$ with respect to $g$.
\end{lemma}
\begin{proof}
    We compute $\Delta_g K$ since 
    \begin{equation}
        \label{lapeq}
        \Delta_g g(K,K) = 2g(\Delta_g K,K)+2g(\nabla^g K,\nabla^g K)
    \end{equation}
    holds. Fix $p\in M$ and take an orthonormal frame $\{e_1,\dots,e_m\}$ around $p$ with respect to $g$ such that $(\nabla^{g}e_1)_p=\cdot\cdot\cdot=(\nabla^{g}e_m)_p=0$. 
    Then, at $p$ we have
    \begin{equation*}
        \begin{split}
            (\Delta_g K)(e_i,e_j) &= \sum_{l=1}^m(\nabla^g_{e_l}(\nabla^gK))(e_i,e_j;e_l)\\
            &= \sum_{l=1}^m(\nabla^g_{e_l}(\nabla^gK))(e_l,e_j;e_i)\\
            &= \sum_{l=1}^m(\nabla^g_{e_i}(\nabla^gK))(e_l,e_l;e_j)+\sum_{l=1}^m(R^g(e_l,e_i)K)(e_l,e_j)\\
            &= \sum_{l=1}^m(R^g(e_l,e_i)K)(e_l,e_j).
        \end{split}
    \end{equation*}
    Here, the last equality follows from 
    \begin{equation*}
    \begin{split}
        \sum_{l=1}^m(\nabla^g_XK)(e_l,e_l)&=\sum_{k,l=1}^mg((\nabla^g_XK)(e_l,e_l),e_k)e_k\\
        &=\sum_{k,l=1}^mg((\nabla^g_XK)(e_l,e_k),e_l)e_k\\
        &=\sum_{k=1}^m(\nabla^g_X\eta)(e_k)e_k\\
        &=\nabla^g_X T,
    \end{split}
    \end{equation*}
    where $\eta(X)=g(X,T)$, $X\in\Gamma(TM)$. 
    Since $p\in M$ was arbitrary, equation (\ref{lapeq}) implies the desired result.
\end{proof}
\begin{theorem}
\label{main2}
    Let $(M^m,g,\nabla)$ be a statistical manifold of constant curvature $\lambda$ satisfying the semi-equiaffine condition.
    Assume that $(M,g)$ is a complete simply-connected Riemannian manifold of constant sectional curvature $c$. 
    The statistical structure $(g,\nabla)$ is determined as follows$:$
    \renewcommand{\labelenumi}{$(\theenumi)$}
    \begin{enumerate}
        \item If $(M,g)=(S^m(c),g_0)$, then we have $\nabla=\nabla^{g}$ and $\lambda=c$.
        \item If $(M,g)=(\mathbb{R}^m,g_0)$, then we have $\lambda\cdot m(m-1)=g(T,T)-g(K,K)$, and if we let
        \begin{equation}
            \label{determst}
            K=\sum_{i,j,k=1}^mK^k_{~ij}\frac{\partial}{\partial x^k}\otimes dx^i\otimes dx^j, 
        \end{equation}
        each element $K^k_{~ij}$, $i,j,k\in\{1,2,\dots,m\}$ is a constant.
        Here, $\{x^1,\dots,x^m\}$ is a coordinate system on $\mathbb{R}^m$ such that $\nabla^{g_0}dx^i=0$.
        \item If $(M,g)=(\mathbb{H}^m(c),g_0)$, then we have $\nabla=\nabla^{g}$ and $\lambda=-c$.
    \end{enumerate}
\end{theorem}
\begin{proof}
    We have
    \begin{equation}
        (R-R^g)(X,Y)Z=(\lambda-c)(g(Y,Z)X-g(X,Z)Y),\quad X,Y,Z\in\Gamma(TM).
    \end{equation}
    In the cases $(1)$ and $(3)$, by Theorem \ref{determtche} we have $T=0$ on $M$. 
    By Theorem $4.4$ of \cite{OPOZDA2016134}, it follows that $\nabla=\nabla^g$.
    
    By equation $(39)$ in \cite{MR3422914}, on a statistical manifold of constant curvature $\lambda$, the relation
    \begin{equation*}
        \lambda\cdot m(m-1)=\widehat{\rho}+g(T,T)-g(K,K)
    \end{equation*}
    holds, where $\widehat{\rho}$ is the scalar curvature of $g$. 
    Thus, in case $(2)$, the function $g(K,K)$ is constant on $M$, yielding $\lambda\cdot m(m-1)=g(T,T)-g(K,K)$. 
    Applying Lemma \ref{lapC}, we have 
    \begin{equation*}
        2g(\nabla^gK,\nabla^gK)=0,
    \end{equation*}
    which implies $\nabla^gK=0$. Therefore, we have $(\ref{determst})$.
\end{proof}




\begin{thebibliography}{9}
\bibitem{Am} 
S. Amari and H. Nagaoka, Differential Geometry of Smooth Families of Probability Distributions, Technical Report METR 82-7, University of Tokyo, 1982.
\bibitem{Berger1971LeSD}
M. Berger, P. Gauduchon, and E. Mazet, \textit{Le {S}pectre d'une {V}ariete {R}iemannienne}, Lect. Notes Math. \textbf{194} (1971), 141--241.
\bibitem{10.2969/jmsj/03620295}
B.Y. Chen and T. Nagano, \textit{Harmonic metrics, harmonic tensors, and Gauss maps}, J. Math. Soc. Japan \textbf{36} (1984), 295--313.
\bibitem{DPN}
S. Deshmukh, P. Peska, and N. Bin Turki, \textit{Geodesic Vector Fields on a Riemannian Manifold}, Mathematics \textbf{80} (2020).
\bibitem{EL}
J. Eells and L. Lemaire, \textit{Selected topics in harmonic maps}, CBMS \textbf{50}, Amer. Math. Soc (1983).
\bibitem{FU}
H. Furuhata and R. Ueno, \textit{A {V}ariation {P}roblem for {M}appings between {S}tatistical {M}anifolds}, Results Math. \textbf{80} (2025).
\bibitem{Furuhata2006TheCM}
H. Furuhata and L. Vrancken, \textit{The {C}enter {M}ap of an {A}ffine {I}mmersion}, Results Math. \textbf{49} (2006), 201--217.
\bibitem{Hopf1926}
H. Hopf, \textit{Zum {C}lifford-{K}leinschen {R}aumproblem}, Math. Ann. \textbf{95} (1926), 313--339.
\bibitem{Killing1891}
W. Killing, \textit{Ueber die {C}lifford-{K}lein'schen {R}aumformen}, Math. Ann. \textbf{39} (1891), 257--278.
\bibitem{MR1030710}
T. Kurose, \textit{Dual connections and affine geometry}, Math. Z. \textbf{203} (1990), 115--121.
\bibitem{Liu1995TheCT}
H. Liu and C.P. Wang, \textit{The {C}entroaffine Tchebychev Operator}, Results Math. \textbf{27} (1995), 77--92.
\bibitem{MR2280051}
H. Matsuzoe, J. Takeuchi, and S. Amari, \textit{Equiaffine structures on statistical manifolds and {B}ayesian statistics}, Differ. Geom. Appl. \textbf{24} (2006), 567--578.
\bibitem{NomizuSasaki}
K. Nomizu and T. Sasaki, Affine differential geometry, Cambridge University Press, Cambridge, 1994.
\bibitem{MR3422914}
B. Opozda, \textit{Bochner's technique for statistical structures}, Ann. Global Anal. Geom. \textbf{48} (2015), 357--395.
\bibitem{OPOZDA2016134}
B. Opozda, \textit{A sectional curvature for statistical structures}, Linear Algebra Appl. \textbf{497} (2016), 134--161.
\bibitem{MR1286138}
C.P. Wang, \textit{Centroaffine minimal hypersurfaces in {${\bf R}^{n+1}$}}, Geom. Dedicata \textbf{51} (1994), 63--74.
\bibitem{MR124854}
K. Yano and T. Nagano, \textit{On geodesic vector fields in a compact orientable {R}iemannian space}, Comment. Math. Helv. \textbf{35} (1961), 55--64.

\end{thebibliography}
\end{document}